\documentclass[11pt]{amsart}%

\usepackage{amsfonts}
\usepackage{amsmath}
\usepackage{amssymb}
\usepackage{graphicx}
\usepackage{color,graphicx}%
\providecommand{\U}[1]{\protect\rule{.1in}{.1in}}
\newtheorem{theorem}{Theorem}

\newtheorem{definition}[theorem]{Definition}

\newtheorem{lemma}[theorem]{Lemma}

\newtheorem{proposition}[theorem]{Proposition}
\newtheorem{remark}[theorem]{Remark}

\begin{document}
\title{Standing waves of fixed period for $n+1$ vortex filaments}
\author{Walter Craig}
\address{Department of Mathematics \& Statistics, McMaster University, Hamilton Ontario
L8S 4K1 \textsc{Canada}}
\author{Carlos Garc\'{\i}a-Azpeitia}
\address{Departamento de Matem\'{a}ticas, Facultad de Ciencias, Universidad Nacional
Aut\'{o}noma de M\'{e}xico, 04510 M\'{e}xico DF, M\'{e}xico}
\thanks{Walter Craig is deceased}
\subjclass[2010]{35B10, 35B32}
\keywords{Vortex filaments. Periodic solutions. Bifurcation.}

\begin{abstract}
The $n+1$ vortex filament problem has explicit solutions consisting of $n$
parallel filaments of equal circulation in the form of nested polygons
uniformly rotating around a central filament which has circulation of opposite
sign. We show that when the relation between temporal and spatial periods is
fixed at certain rational numbers, these configurations have an infinite
number of homographic time dependent standing wave patterns that bifurcate
from these uniformly rotating central configurations.

\end{abstract}
\maketitle

\section*{Introduction}

In reference \cite{KMD95}, a model system of equations was derived for the
interaction of near-parallel vortex filaments. The model considers vortex
filaments in $\mathbb{R}^{3}$ to be coordinatized by curves $(u_{j}%
(t,s),s)\in\mathbb{C}\times\mathbb{R}$ for $j=0,\dots,n$ that describe the
positions of $n+1$ vertically oriented vortex filaments. Different aspects of
this problem have been investigated in
\cite{BaMi,BaMi12,GaCr15,Ga17,GaIz12,Po03,Ne01} and references therein. In
this article we study central configurations of $n+1$ vortex filaments with
$n$ filaments of equal circulation and one filament of opposite circulation.

Let $u_{j}(t,s)$ for $j=1,...,n$ be the positions of the $n$ filaments of
circulation $1$ and $u_{0}(t,s)$ the filament of circulation $-\kappa$ with
$\kappa>0$. A homographic standing wave of the $n+1$ vortex filament problem
with fixed period is a solution of the form
\begin{equation}
u_{j}(t,s)=ae^{i\omega t}\left(  a_{j}+a_{j}u(t/q,s)\right)  ~, \label{SW}%
\end{equation}
where $\omega=-a^{-2}$ is real, $q$ is an integer and $u(t,s)$ is a complex
$2\pi$-periodic function in $t$ and $s$.

The complex numbers $a_{j}\in\mathbb{C}$ for $j=0,...,n$ lie in a central
configuration with $a_{0}=0$. That is, the complex numbers $a_{j}$ satisfy
\begin{equation}
0=\sum_{i=1}^{n}\frac{a_{i}}{\left\vert a_{i}\right\vert ^{2}},\qquad
-a_{j}=\sum_{i=1(i\neq j)}^{n}\frac{a_{j}-a_{i}}{\left\vert a_{j}%
-a_{i}\right\vert ^{2}}-\kappa\frac{a_{j}}{\left\vert a_{j}\right\vert ^{2}},
\label{cc}%
\end{equation}
for $j=1,...,n$. There are many configurations that satisfy \eqref{cc}, for
example in the form of nested polygons. In particular, an explicit solution of
(\ref{cc}) is given by the regular polygon
\[
a_{j}=\left(  \kappa-\left(  n-1\right)  /2\right)  ^{-1/2}e^{ij\zeta}%
,\qquad\zeta=2\pi/n,
\]
if $\kappa>\left(  n-1\right)  /2$.

Setting $u=0$ in equation \eqref{SW} corresponds to the family of homographic
solutions for which $n$ straight parallel filaments rotate around the central
filament with uniform frequency $\omega$ and amplitude $a$. The standing waves
of the title of this article correspond to non-trivial $2\pi$-periodic
solutions $u(t,s)$ of the equation $Lu+g(u)=0$, where $L$ is the linear
operator
\begin{equation}
L(\omega)u:=-\left(  i/q\right)  \partial_{t}u-\partial_{s}^{2}u+\omega\left(
u+\bar{u}\right)  \text{,}%
\end{equation}
and $g$ is an analytic nonlinearity describing the horizontal vortex filament
interaction. Our goal is to construct standing wave solutions that bifurcate
from the initial configuration $u=0$, for which the frequency $\omega$ is the
bifurcation parameter. The solution given by (\ref{SW}) with a $2\pi$-periodic
function $u$ has fixed spatial period $s\in\lbrack0,2\pi)$ and temporal period
$t\in\lbrack0,2\pi q)$ in a frame of reference that is rotating with frequency
$\omega$, i.e. the solution is periodic or quasiperiodic with the two temporal
frequencies $\omega$ and $1/q$ when observed in a stationary reference frame.
The main theorem is as follows.

\begin{theorem}
\label{1}Let $q$ be an integer. For each $k_{0}\in\mathbb{N}$, there is a
local continuum of $2\pi$-periodic solutions $u$ bifurcating from the
unperturbed configuration with $u=0$ and initial frequency
\begin{equation}
\omega_{0}=-\frac{1}{q}\left(  1-\frac{1}{2k_{0}^{2}q}\right)  \text{.}
\label{Om}%
\end{equation}
The local bifurcation $(u,\omega)$ consists of standing waves satisfying the
estimates
\[
u(t,s)=b\left[  \cos j_{0}t+i\left(  1-k_{0}^{-2}/q\right)  \sin
j_{0}t\right]  \cos k_{0}s+\mathcal{O}(b^{2}),
\]
with $\omega=\omega_{0}+\mathcal{O}(b^{2})$ and $j_{0}=qk_{0}^{2}-1$, where
$b\in\lbrack0,b_{0}]$ gives a local parameterization of the bifurcation curve.
Furthermore, these solutions satisfy the following symmetries
\begin{equation}
u(t,s)=\bar{u}(-t,s)=u(t,-s)\text{.}\nonumber
\end{equation}

\end{theorem}

\begin{figure}[h]
\begin{center}
\resizebox{13cm}{!}{
\includegraphics{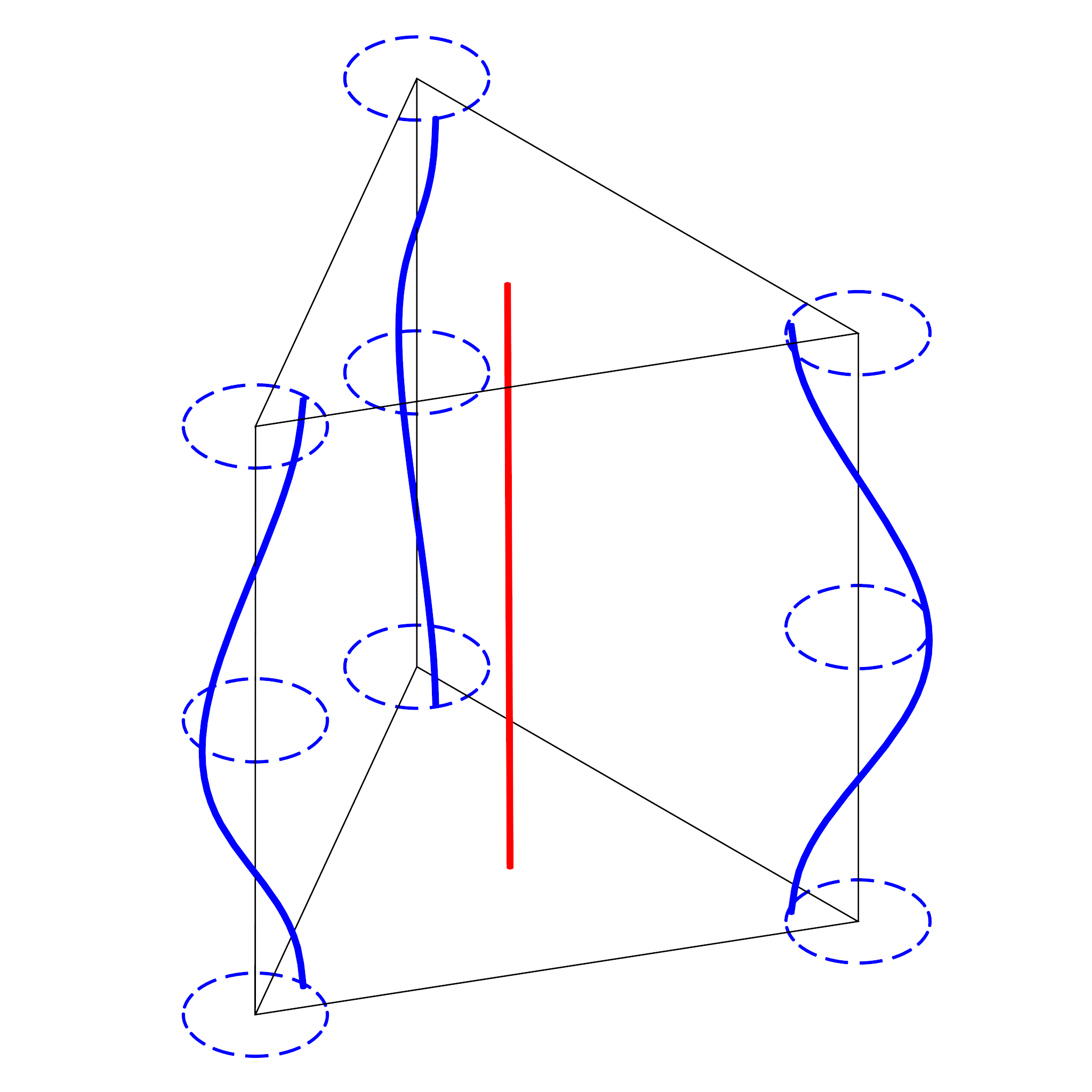}
\includegraphics{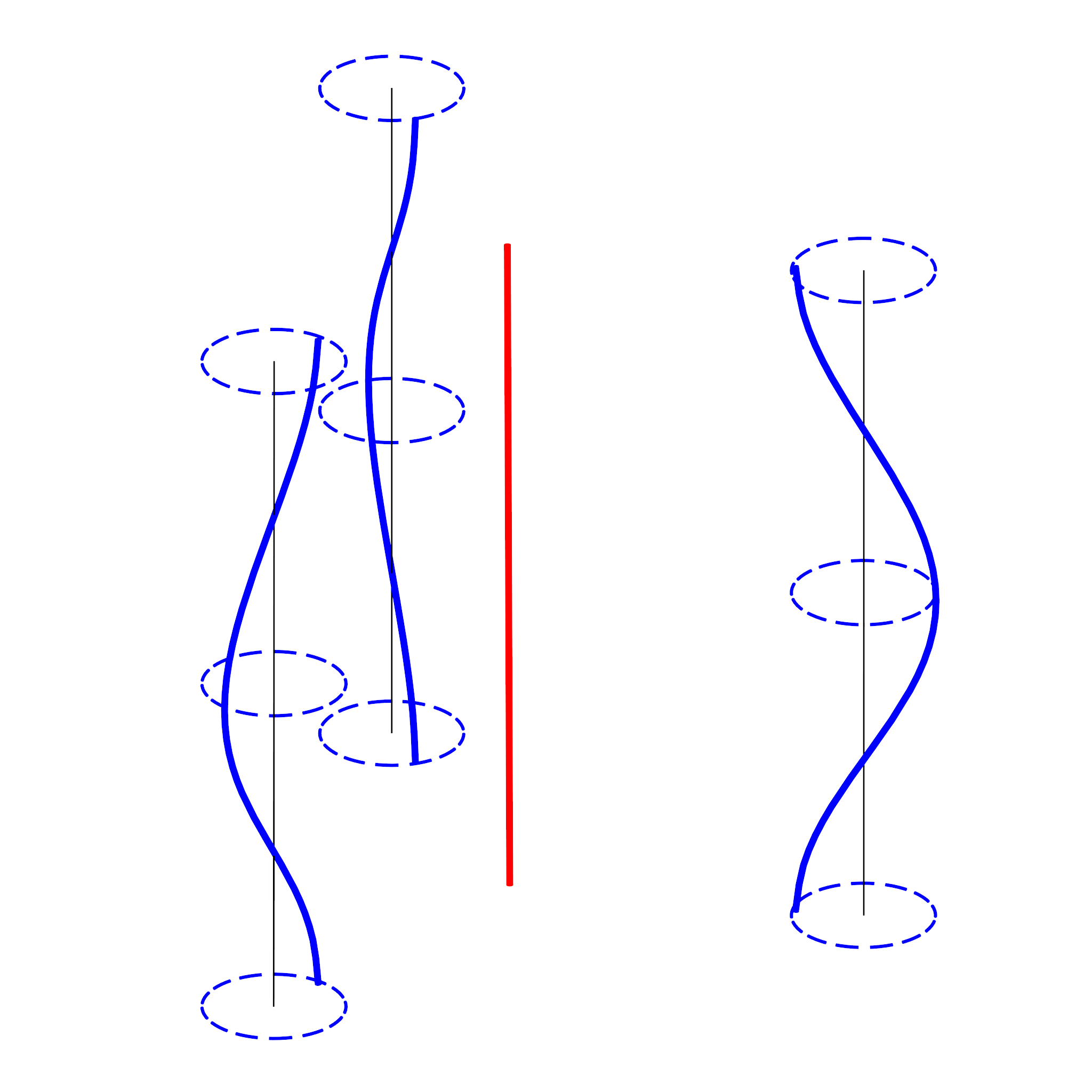}
}
\end{center}
\caption{Illustration of the standing waves obtained in Theorem 1 for the case of $n=3$ vortex filaments of
equal circulation (blue) and one vortex filament of opposite circulation (red).}
\end{figure}
Therefore, for any central configuration $a_{j}$ satisfying (\ref{cc}), the
previous theorem gives homographic solutions of the form (\ref{SW}). The
periodic solutions $u$ are special in that the ratio of their temporal and
spatial periods are rational. In reference \cite{GaCr15} we studied the case
of irrational ratios, which is a small divisor problem for a nonlinear partial
differential equation which requires techniques related to KAM theory even for
the case of constructions of periodic solutions. Our approach is parallel to
that of the semilinear wave and beam equation in one dimension, where time
periodic solutions with rational periods (free vibrations) were shown to exist
in \cite{AmZe80,Ar,Ki79,Ki00,Ra78}, and later for irrational periods in
\cite{Ba,Ra}. On the other hand, time periodic solutions bifurcating from
stationary solutions with irrational periods is a small divisor problem, for
which constructions of solutions by Nash-Moser methods came much later in
\cite{Be07,Bo95,CrWa93}, and references therein.

In the present analysis the ratio of the periodic solution is rational and the
small divisor problem does not occur. The key element of the proof consists on
the fact that for special temporal frequencies, given by $1/q$, the
Schr\"{o}dinger operator $L(\omega)$, when restricted to the orthogonal
complement of the null space, has a bounded inverse in the set of frequencies
$\omega_{0}\in$ $(-1/q,0)$. Unlike in semilinear wave and beam equations, our
equation is a genuine Hamiltonian PDE represented by a Schr\"{o}dinger
operator $L(\omega)$ which does not have the regularity that is usually
obtained in other equations, i.e. our result can be obtained only in a narrow
set of parameters where $L(\omega_{0})$ has a nontrivial kernel. This is also
the case of the counter-rotating vortex filament pair studied in \cite{Ga17},
but this is the first time that periodic solutions without small divisors are
obtained in a genuine non-linear Hamiltonian PDE using this method.

In section 1, we set up a Lyapunov-Schmidt reduction to prove the existence of
standing waves. In section 2 we solve the range equation for $\omega_{0}\in$
$(-1/q,0)$ using the contracting mapping theorem. In section 3 we use the
symmetries of the problem to solve the bifurcation equation by means of the
Crandall-Rabinowitz theorem.

\section{Setting the problem}

From \cite{KMD95} the system of model equations for the dynamics of $n+1$
near-parallel vortex filaments, with circulations $\Gamma_{0}=-\kappa$ and
$\Gamma_{j}=1$ for $j=1,...,n$, is given by
\begin{equation}
\partial_{t}u_{j}=i\left(  \Gamma_{j}\partial_{ss}u_{j}+\sum_{i=0~(i\not =%
j)}^{n}\Gamma_{j}\frac{u_{j}-u_{i}}{\left\vert u_{j}-u_{i}\right\vert ^{2}%
}\right)  ,\qquad j=0,...,n. \label{fp}%
\end{equation}
Homographic solutions of the $n+1$ filaments are particular solutions of the
form
\[
u_{j}(t,s)=w(t,s)a_{j}~,
\]
where $w(t,s)$ is a complex valued function and where $a_{j}$'s are complex
numbers satisfying the condition of a central configuration. In this class of
solutions the shape of the intersections of the filaments with a horizontal
complex plane is homographic with the shape of their intersection with any
other horizontal plane $\{x_{3}=c\}$ for any $c$ and at any time $t$.

For a general central configuration
\begin{equation}
-a_{j}=\sum_{i=0(i\not =j)}^{n}\Gamma_{i}\frac{a_{j}-a_{i}}{|a_{j}-a_{i}|^{2}%
}~,\qquad j=0,...,n, \label{Eqn:CentralConfiguration}%
\end{equation}
homographic solutions satisfy the system of equations (\ref{fp}) if $w(t,s)$
solves the system of equations
\[
a_{j}\partial_{t}w(t,s)=i\left(  \Gamma_{j}a_{j}\partial_{ss}w(t,s)-\frac
{w(t,s)}{\left\vert w(t,s)\right\vert ^{2}}a_{j}\right)  ,\qquad
j=0,...,n,\text{.}%
\]
In the particular case that $a_{0}=0$ in the central configuration, the
condition for the configuration $a_{j}$ becomes (\ref{cc}) and the system of
equations is satisfied by solutions of the simple equation,
\begin{equation}
\partial_{t}w=i\left(  \partial_{ss}w-\frac{w}{\left\vert w\right\vert ^{2}%
}\right)  \text{.} \label{pde}%
\end{equation}
Therefore, $u_{j}(t,s)=w(t,s)a_{j}$ is an homographic solution of the vortex
filament problem if the configuration $a_{j}$ satisfies (\ref{cc}) and $w$ is
a solution of the equation (\ref{pde})

A particular solution of (\ref{cc}) is given by a regular polygon
$a_{j}=re^{ij\zeta}$ with radius $r=\left(  \kappa-\left(  n-1\right)
/2\right)  ^{-1/2}$ if $\kappa>\left(  n-1\right)  /2$, because
\[
\sum_{i=1(i\neq j)}^{n}\frac{a_{j}-a_{i}}{\left\vert a_{j}-a_{i}\right\vert
^{2}}-\kappa\frac{a_{j}}{\left\vert a_{j}\right\vert ^{2}}=-\left(
\kappa-\frac{n-1}{2}\right)  \frac{a_{j}}{r^{2}}=-a_{j}~.
\]
Also, there are other solutions of (\ref{cc}) corresponding to nested polygons.

Equation \eqref{pde} has the set of solution $w=ae^{i\omega t}$ with
\[
\omega=-a^{-2}<0~,
\]
that corresponds to $n$ vortex filaments uniformly rotating in the central
configuration $a_{j}$ with amplitude $a$ and frequency $\omega$. We look for
bifurcation of solutions of the equation \eqref{pde} of the form
\begin{equation}
w(t,s)=ae^{i\omega t}(1+u(t/q,s)), \label{cov}%
\end{equation}
where $q$ is an integer and $u(t,s)$ is $2\pi$-periodic in $t$ and $s$. This
is a solution that has fixed temporal and spatial periodicity when viewed in a
coordinate frame rotating about the $x_{3}$-axis with frequency $\omega$. When
$u=0$ the solution corresponds to $n$ vortex filaments uniformly rotating in
the central configuration $a_{j}$. The equation \eqref{pde} for a perturbation
from this configuration is
\begin{equation}
\left(  i/q\right)  \partial_{t}u=-u_{ss}+\omega(u+\bar{u})+g(u,\overline
{u})~\text{,} \label{pdev}%
\end{equation}
where the nonlinearity $g$ is given by
\[
g(u,\bar{u})=\omega\frac{\bar{u}^{2}}{1+\bar{u}}=\omega\sum_{n=2}^{\infty
}(-1)^{n}\bar{u}^{n}~.
\]

In order to simplify the analysis of symmetries, the equation is represented
in real coordinates $u(\tau,s)=(x(\tau,s),y(\tau,s))\in\mathbb{R}^{2}$, i.e.,
the equation is equivalent to
\[
Lu+g(u)=0,
\]
where $g(u)=\mathcal{O}\left(  \left\vert u\right\vert ^{2}\right)  $ is
analytic for $\left\vert (x,y)\right\vert <1$ and $L$ is the linear operator
\begin{equation}
Lu:=-\left(  1/q\right)  J\partial_{t}u-\partial_{s}^{2}u+\omega\left(
I+R\right)  u\text{,}%
\end{equation}
where $R=diag(1,-1)$.

We define the Hilbert space $L^{2}({\mathbb{T}}^{2};\mathbb{R}^{2})$, with the
inner product%
\[
\left\langle u_{1},u_{2}\right\rangle =\frac{1}{(2\pi)^{2}}\int_{{\mathbb{T}%
}^{2}}u_{1}\cdot u_{2}\,dtds\text{.}%
\]
A function $u\in L^{2}$ can be written in a Fourier basis as
\[
u=\sum_{(j,k)\in\mathbb{Z}^{2}}u_{j,k}e^{i(jt+ks)},\qquad u_{j,k}=\bar
{u}_{-j,-k}\in\mathbb{C}^{2}.
\]
The Sobolev space $H^{s}$ is the usual subspace of functions in $L^{2}$ with
bounded norm
\[
\left\Vert u\right\Vert _{H^{s}}^{2}=\sum_{(j,k)\in\mathbb{Z}^{2}}\left\vert
u_{j,k}\right\vert ^{2}\left(  j^{2}+k^{2}+1\right)  ^{s}\text{.}%
\]
This space has the Banach algebra property for $s>1$,%
\[
\left\Vert uv\right\Vert _{H^{s}}\leq C\left\Vert u\right\Vert _{H^{s}%
}\left\Vert v\right\Vert _{H^{s}}~.
\]
The Banach algebra property implies that the nonlinear operator
$g(u)=\mathcal{O}(\left\Vert u\right\Vert _{H^{s}}^{2})$ is well defined and
continuous for $\left\Vert u\right\Vert _{H^{s}}<1$.

The linear operator $L:D(L)\rightarrow H^{s}$ is continuous when the domain%
\[
D(L)=\{u\in H^{s}:Lu\in H^{s}\}~,
\]
is completed under the graph norm
\[
\left\Vert u\right\Vert _{L}^{2}=\left\Vert Lu\right\Vert _{H^{s}}%
^{2}+\left\Vert u\right\Vert _{H^{s}}^{2}~.
\]
In Fourier basis, the operator $L:D(L)\rightarrow H^{s}$ is given by
\[
Lu=\sum_{(j,k)\in\mathbb{Z}^{2}}M_{j,k}u_{j,k}e^{i(jt+ks)}~,
\]
where
\[
M_{j,k}=\left(
\begin{array}
[c]{cc}%
k^{2}+2\omega & i\left(  j/q\right) \\
-i\left(  j/q\right)  & k^{2}%
\end{array}
\right)  \text{.}%
\]
Then, the eigenvalues and eigenvectors of $L$ are
\begin{align}
\lambda_{j,k,l}  &  =k^{2}+\omega+l\sqrt{\left(  j/q\right)  ^{2}+\omega^{2}%
}~,\\
e_{j,k,l}  &  =\left(
\begin{array}
[c]{c}%
-\omega-l\sqrt{\left(  j/q\right)  ^{2}+\omega^{2}}\\
i\left(  j/q\right)
\end{array}
\right)  ,
\end{align}
for $(j,k,l)\in\mathbb{Z}^{2}\times\mathbb{Z}_{2}$, where $\mathbb{Z}%
_{2}=\{1,-1\}$ is a group under the product.

The eigenvalue $\lambda_{j,k,1}$ always is positive, and $\lambda
_{j,k,-1}(\omega_{0})=0$ if%

\[
\omega_{0}=\left(  \left(  j/qk\right)  ^{2}-k^{2}\right)  /2<0.
\]
Given that $L(\omega_{0})$ has a nontrivial kernel, we expect bifurcation of
solutions of $L(\omega)u+g(u)=0$ as $\omega$ crosses $\omega_{0}$.

\begin{definition}
We define $N$ as the subset of all lattice points corresponding to zero
eigenvalues,
\[
N(\omega_{0})=\left\{  \left(  j,k,-1\right)  \in\mathbb{Z}^{2}\times
\mathbb{Z}_{2}:\lambda_{j,k,-1}\left(  \omega_{0}\right)  =0\right\}
~\text{.}
\]

\end{definition}

By definition we have that the kernel of $L(\omega_{0})$ is generated by
eigenfunctions $e_{j,k,l}e^{i(jt+ks)}$ with $\left(  j,k,l\right)  \in N$.
Notice that additional sites to $(\pm j_{0},\pm k_{0},-1)$ may be present in
$N(\omega_{0})$ due to resonances. The Lyapunov-Schmidt reduction separates
the kernel and the range equations using the projections
\[
Qu = \sum_{\left(  j,k,l\right)  \in N} u_{j,k,l}e_{j,k,l}e^{i(jt+ks)}
~,\qquad Pu = (I-Q)u ~.
\]
Setting
\[
u = v+w ~, \qquad v = Qu ~, \qquad w = Pu~\text{,}%
\]
the equation $Lu+g(u)=0$ is equivalent to the kernel equation
\begin{equation}
QLQv + Qg(v+w) = 0~,
\end{equation}
and the range equation
\begin{equation}
PLPw + Pg(v+w)= 0~.
\end{equation}

\section{The range equation}

In this section, the range equation is solved as a fixed point $w(\omega,v)\in
H^{s}$ of the operator
\[
Kw=-\left(  PLP\right)  ^{-1}g(w+v,\omega)~\text{.}%
\]
The local solution $w=w(\omega,v)$ is provided by an application of the
contraction mapping theorem, where we only need to prove that $\left(
PLP\right)  ^{-1}:PH^{s}\rightarrow PH^{s}$ is well defined and bounded. For
this, we will establish bound estimates in the eigenvalues $\lambda_{j,k,l}$.

For $l=1$, we clearly have%
\[
\lambda_{j,k,1}=k^{2}+\omega+\sqrt{\left(  j/q\right)  ^{2}+\omega^{2}}\gtrsim
k^{2}+\left\vert j\right\vert ~\text{.}%
\]
For $l=-1$, we have the following estimate,

\begin{lemma}
For $2\varepsilon< \left\vert \omega\right\vert < 1/q-2\varepsilon$, we have
\begin{equation}
\left\vert \lambda_{j,k,-1}(\omega)\right\vert \gtrsim\varepsilon\text{ for
}\left(  j,k,l\right)  \in N^{c}\text{.}%
\end{equation}

\end{lemma}

\begin{proof}
In the case $\left\vert j\right\vert /q\neq k^{2}$, the inequality $\left\vert
k^{2}-\left\vert j\right\vert /q\right\vert \geq1/q$ holds and
\[
\left\vert \lambda_{j,k,-1}(\omega_{0})\right\vert \geq\left\vert
k^{2}-\left\vert j\right\vert /q\right\vert -\left\vert \left\vert
j\right\vert /q+\omega-\sqrt{\left(  j/q\right)  ^{2}+\omega^{2}}\right\vert
~.
\]
Since $\lim_{x\rightarrow\infty}\left(  x+\omega-\sqrt{x^{2}+\omega^{2}%
}\right)  =\omega$, then
\[
\left\vert \left\vert j\right\vert /q+\omega-\sqrt{\left(  j/q\right)
^{2}+\omega^{2}}\right\vert <\left\vert \omega\right\vert +\varepsilon~,
\]
for $\left\vert k\right\vert +\left\vert j\right\vert \geq M$ with $M$ big
enough. Therefore,
\[
\left\vert \lambda_{j,k,-1}(\omega)\right\vert \geq\left\vert k^{2}-\left\vert
j\right\vert /q\right\vert -\left\vert \omega\right\vert -\varepsilon\geq
\frac{1}{q}-\left\vert \omega\right\vert -\varepsilon\geq\varepsilon~.
\]
In the case $\left\vert j\right\vert /q=k^{2}$, then
\[
\left\vert \lambda_{j,k,-1}(\omega)\right\vert =\left\vert k^{2}+\omega
-\sqrt{k^{2}+\omega^{2}}\right\vert \geq\left\vert \omega\right\vert
-\varepsilon\geq\varepsilon~.
\]
for $\left\vert k\right\vert $ big enough. In both cases we have that
$\left\vert \lambda_{j,k,-1}(\omega)\right\vert \geq\varepsilon$ if
$\left\vert k\right\vert +\left\vert j\right\vert \geq M$ with $M$ big enough.
We conclude that the estimate holds except by a finite number of points
$\left(  j,k\right)  \in\mathbb{Z}^{2}$. Therefore, there is a constant $c$
such that the estimate $\left\vert \lambda_{j,k,-1}(\omega)\right\vert \geq
c\varepsilon$ holds for all $(j,k,-1)\in N^{c}$.
\end{proof}

From the previous estimates we have that $\left(  PLP\right)  ^{-1}$ is a
bounded operator with
\[
\left\Vert \left(  PLP\right)  ^{-1}w\right\Vert _{H^{s}}\lesssim
\varepsilon^{-1}\left\Vert w\right\Vert _{H^{s}}\text{.}%
\]

\begin{proposition}
\label{2}Assume $2\varepsilon<\left\vert \omega\right\vert <1/q-2\varepsilon$.
There is a unique continuous solution $w(v,\omega)\in H^{s}$ of the range
equation defined for $(v,\omega)$ in a small neighborhood of $(0,\omega
)\in\ker L(a_{0})\times\mathbb{R}$ such that
\begin{equation}
\left\Vert w(v,\omega)\right\Vert _{H^{s}}\lesssim\varepsilon^{-1}\left\Vert
v\right\Vert ^{2}\text{,}%
\end{equation}
for small $\varepsilon$.
\end{proposition}

\begin{proof}
By the Banach algebra property of $H^{s}$, the operator
\[
g(w)=\mathcal{O}(\left\Vert w\right\Vert _{H^{s}}^{2}):B_{\rho}\rightarrow
H^{s}%
\]
is well define in the domain $B_{\rho}=\{w\in H^{s}:\left\Vert w\right\Vert
_{H^{s}}<\rho\}$ for $\rho<1$. We can chose a small enough $\varepsilon$ such
that the hypothesis of the previous lemma hold true. Therefore,
\begin{align*}
Kw  &  =-\left(  PLP\right)  ^{-1}g(w+v,\omega)=\mathcal{O}(\varepsilon
^{-1}\left\Vert w\right\Vert _{H^{s}}^{2})\\
K  &  :B_{\rho}\subset PH^{s}\rightarrow PH^{s}\text{,}%
\end{align*}
is well defined and continuous. Moreover, it is a contraction for $\rho$ of
order $\rho=\mathcal{O}(\varepsilon)$. By the contraction mapping theorem,
there is a unique continuous fixed point $w(v,\omega)\in B_{\rho}$. The
estimate $\left\Vert w(v,\omega)\right\Vert _{H^{s}}\leq\varepsilon
^{-1}\left\Vert v\right\Vert ^{2}$ is obtained from
\[
\left\Vert Kw\right\Vert _{H^{s}}\lesssim\varepsilon^{-1}\left(  \left\Vert
w\right\Vert _{H^{s}}^{2}+\left\Vert v\right\Vert ^{2}\right)  .
\]

\end{proof}

\begin{remark}
Since $\left(  PLP\right)  ^{-1}$ is continuous but not compact, we do not
automatically obtain the regularity of the solutions by bootstrapping
arguments. Instead,the regularity is obtained using the Sobolev embedding
$H^{s}\subset C^{2}$ for $s\geq3$.
\end{remark}

\section{The bifurcation equation}

\begin{proposition}
For $k_{0}\in\mathbb{N}$, we define
\begin{equation}
\omega_{0}=-\frac{1}{q}\left(  1-\frac{1}{2qk_{0}^{2}}\right)  ,\qquad
j_{0}=qk_{0}^{2}-1.
\end{equation}
For these frequencies we have $\omega_{0}\in\left(  -1/q,0\right)  $ and
\[
N(\omega_{0})=\{(0,0,-1),\left(  \pm j_{0},\pm k_{0},-1\right)  \}.
\]

\end{proposition}

\begin{proof}
Since $\lambda_{j,k,-1}=k^{2}+\omega-\sqrt{\left(  j/q\right)  ^{2}+\omega
^{2}}$, then $\lambda_{j,0,-1}(\omega)=0$ only if $j=0$. For $k_{0}%
\in\mathbb{N}^{+}$, the condition $\lambda_{j_{0},k_{0},-1}(\omega_{0})=0$
$\ $is satisfied only if
\[
\omega_{0}=\left(  \left(  j_{0}/qk_{0}\right)  ^{2}-k_{0}^{2}\right)
/2\text{.}%
\]
In addition, the condition $\omega_{0}\in\left(  -1/q,0\right)  $ holds if an
only if the lattice point $\left(  j_{0},k_{0}\right)  \in\mathbb{N}^{2}$
satisfies $j_{0}=qk_{0}^{2}-1$. In this case
\[
\omega_{0}=\frac{1}{2}\left(  \left(  k_{0}-\frac{1}{qk_{0}}\right)
^{2}-k_{0}^{2}\right)  =-\frac{1}{q}\left(  1-\frac{1}{2qk_{0}^{2}}\right)  ,
\]
then the frequency $\omega_{0}$ is determined uniquely for each point $\left(
j_{0},k_{0}\right)  \in\mathbb{N}^{2}$ because $\omega_{0}$ is decreasing in
$k_{0}$. Therefore, we have that $(0,0,-1)$ and $\left(  \pm j_{0},\pm
k_{0},-1\right)  $ are the only elements in $N(\omega_{0})$.
\end{proof}

Since $\ker L(\omega_{0})$ has dimension $5$ for $\omega_{0}\in\left(
-1/q,0\right)  $, we need to reduce the bifurcation equation to a subspace of
dimension one in order to apply the Crandall-Rabinowitz theorem. This is
attained by exploiting the equivariance of the system \eqref{pdev} under the
action of the group $G=O(2)\times O(2)$ given by
\[
\rho(\tau,\sigma)u(t,s)=u(t+\tau,s+\sigma)~,
\]
for the abelian components, and for the reflections,
\[
\rho(\kappa_{1})u(t,s)=Ru(-t,s),\quad\rho(\kappa_{2})u(t,s)=u(t,-s)~,
\]
where $R=diag(1,-1)$. By the uniqueness of $w(v,\omega)$, the bifurcation
equation has the same equivariant properties as the differential equation.
This property is used in the following proposition to reduce the bifurcation
equation to a subspace of dimension one.

\begin{proposition}
\label{3}The bifurcation equation has a local continuum of $2\pi$-periodic
solution bifurcating from $(v,\omega)=(0,\omega_{0})$ with estimates
\begin{equation}
v(t,s)=b\left(
\begin{array}
[c]{c}%
\cos j_{0}t\\
\left(  1-k_{0}^{-2}/q\right)  \sin j_{0}t
\end{array}
\right)  \cos k_{0}s+\mathcal{O}(b^{2})\text{,\qquad}\omega=\omega
_{0}+\mathcal{O}(b^{2}),
\end{equation}
where $b\in\lbrack0,b_{0}]$ gives a parameterization of the local bifurcation,
and symmetries
\begin{equation}
v(t,s)=Rv(-t,s)=v(t,-s)=v(t+\pi/j_{0},s+\pi/k_{0}).
\end{equation}

\end{proposition}

\begin{proof}
In Fourier components
\[
v=\sum_{\left(  j,k,l\right)  \in N}u_{j,k,l}e_{j,k,l}e^{i(jt+ks)},\qquad
u_{j,k,l}=\bar{u}_{-j,-k,l},
\]
the action of the abelian part of the group $G$ is given by
\[
\rho(\varphi)u_{j,k,l}=e^{ij\varphi}u_{j,k,l},\qquad\rho(\theta)u_{j,k,l}%
=e^{ik\theta}u_{j,k,l}~.
\]
Since
\[
e_{j,k,-1}=\left(
\begin{array}
[c]{c}%
-\omega_{0}-\sqrt{\left(  j/q\right)  ^{2}+\omega_{0}^{2}}\\
i\left(  j/q\right)
\end{array}
\right)  =\left(
\begin{array}
[c]{c}%
k^{2}\\
i(j/q)
\end{array}
\right)  ,
\]
then $Re_{j,k,-1}=e_{-j,k,-1}$ and $e_{j,k,-1}=e_{j,-k,-1}$. Therefore, we
have
\[
\rho(\kappa_{1})v=\sum_{\left(  j,k,l\right)  \in N}u_{j,k,l}e_{-j,k,-1}%
e^{i(-jt+ks)}=\sum_{\left(  j,k,l\right)  \in N}u_{-j,k,l}e_{j,k,-1}%
e^{i(jt+ks)},
\]
and%
\[
\rho(\kappa_{2})v=\sum_{\left(  j,k,l\right)  \in N}u_{j,k,l}e_{j,k,-1}%
e^{i(jt-ks)}=\sum_{\left(  j,k,l\right)  \in N}u_{j,-k,l}e_{j,k,-1}%
e^{i(jt+ks)}\text{.}%
\]
Therefore, the action of the reflections in Fourier components is given by%
\[
\rho(\kappa_{1})u_{j,k,l}=u_{-j,k,l}=\bar{u}_{j,-k,l}~,\qquad\rho(\kappa
_{2})u_{j,k,l}=u_{j,-k,l}~.
\]

The irreducible representations under the action of $O(2)\times O(2)$
corresponds to the subspaces%
\[
(u_{j_{0},k_{0},-1},u_{j_{0},-k_{0},-1})\in\mathbb{C}^{2}.
\]
The linear operator $L$ is diagonal in these irreducible representations with
eigenvalue $\lambda_{j_{0},k_{0},-1}$ of complex multiplicity two. The group
\[
S=\left\langle \kappa_{1},\kappa_{2},(\pi/j_{0},\pi/k_{0})\right\rangle
\]
has fixed point space $(u_{j,k,-1},u_{j,-k,-1})=(b,b)$ for $b\in\mathbb{R}$ in
this representation. By setting
\[
\ker L^{S}(\omega_{0}):=\ker L(\omega_{0})\cap\emph{Fix~}(S)~\text{,}%
\]
the bifurcation equation%
\begin{equation}
QLQw+Qg(v+w(v,\omega)):\ker L^{S}(\omega_{0})\times\mathbb{R}\rightarrow\ker
L^{S}(\omega_{0}) \label{BE}%
\end{equation}
is well defined by the equivariance properties. Moreover, since $u_{0,0,-1}%
$\emph{ }is not fixed by the subgroup $S$, then $\ker L^{S}(\omega_{0})$ is
generated by the simple eigenfunction
\[
\sum_{j=\pm j_{0},k=\pm k_{0}}e_{j,k,-1}e^{i(jt+ks)}=4\left(
\begin{array}
[c]{c}%
k_{0}^{2}\cos j_{0}t\\
j_{0}/q\sin j_{0}t
\end{array}
\right)  \cos k_{0}s\text{.}%
\]

Since $\ker L^{S}(\omega_{0})$ has dimension one, the local bifurcation for
$\omega$ close to $\omega_{0}$ follows from the Crandall-Rabinowitz theorem
applied to the bifurcation equation (\ref{BE}). It is only necessary to verify
that $\partial_{\omega}L(\omega)f$ is not in the range of $L$ for $f\in\ker
L^{S}(\omega_{0})$, which follows from the fact that
\[
\partial_{\omega}L(\omega)f=\left(  I+R\right)  f~.
\]
The estimates $\omega=\omega_{0}+\mathcal{O}(b)$ and
\[
v(t,s)=b\left(
\begin{array}
[c]{c}%
\cos j_{0}t\\
\left(  1-k_{0}^{-2}/q\right)  \sin j_{0}t
\end{array}
\right)  \cos k_{0}s+\mathcal{O}(b^{2})
\]
are consequence of the Crandall-Rabinowitz estimates. Moreover, the
${\mathbb{S}}^{1}$-action of the element $\varphi=\pi$ in the kernel generated
is given by $\rho(\varphi)=-1$. This symmetry implies that the bifurcation
equation is odd and $\omega=\omega_{0}+\mathcal{O}(b^{2})$.
\end{proof}

The main theorem follows from this proposition and the fact that
$u=v+w(v,\omega)$ with
\[
\left\Vert w(v,\omega)\right\Vert _{H^{s}}=\mathcal{O}\left(  \left\Vert
v\right\Vert ^{2}\right)  =\mathcal{O}\left(  b^{2}\right)  .
\]

\vskip0.25cm \textbf{Acknowledgements.} W.C. was partially supported by the
Canada Research Chairs Program and NSERC through grant number 238452--16.
C.G.A was partially supported by a UNAM-PAPIIT project IN115019. We acknowledge the assistance of Ramiro Chavez Tovar with the preparation
of the figure.

\end{document}